\documentclass{amsart}
\usepackage{graphicx} % Required for inserting images
\usepackage{amssymb}
\usepackage{amsmath}
\usepackage{amsthm}
\usepackage{amsfonts}
\usepackage[style=ieee, maxbibnames=99, maxcitenames=99]{biblatex}
\usepackage{breqn}
\usepackage{algorithm} % for pseudocode
\usepackage{algpseudocode} % algorithmic environment
\usepackage{tikz}
\usepackage{hyperref}

\DeclareRobustCommand{\stirlings}{\genfrac\{\}{0pt}{}}

\title{Duality relations of graph polynomials}
\author{Medet Jumadildayev}
\address{Institute of Mathematics and Mathematical Modeling, Almaty, Kazakhstan}
\address{Nazarbayev University, Astana, Kazakhstan}
\email{medetdzhuma@gmail.com}
\subjclass[2020]{05C31, 33C45, 05C38, 05C85}

\newtheorem{thm}{Theorem}[section]
\newtheorem{dfn}[thm]{Definition}
\newtheorem{crl}[thm]{Corollary}
\newtheorem{prp}[thm]{Proposition}
\newtheorem{lm}[thm]{Lemma}

\theoremstyle{remark}
\newtheorem{rmk}[thm]{Remark}

\bibliography{references} 

\begin{document}

\begin{abstract}
    The duality theorem of Lass relates the matching polynomials of a simple graph $G$ with the matching polynomials of its complement $\bar G$. In particular, this relation gives rise to Godsil's result, which offers a nice interpretation of the Lebesgue integral associated with the Hermite orthogonality measure. In this work, we show that similar to matching polynomials, the path-cover polynomials also satisfy duality relations and give combinatorial interpretations of the Lebesgue integral and the inner product in the space of associated Laguerre polynomials. Similar duality relations hold for clique-cover polynomials and chromatic polynomials. As applications, we find an efficient algorithm that computes these graph polynomials for cographs. We also give explicit formulas to compute the number of Hamiltonian paths and cycles in complete multipartite graphs.
\end{abstract}

\maketitle

\section{Introduction}

Let $G = (V, E)$ be a simple graph. Our discussion focuses on the path-cover polynomial $\pi_G(t)$, the matching polynomial $\mu_G(t)$, the clique-cover polynomial $\xi_G(t)$, and the chromatic polynomial $\chi_G(t)$. Denote $\pi_G^+(t)$ as the signed path-cover polynomial, and $\mu_G^+(t)$ as the signed matching polynomial. We give the technical definitions in Section 2. 

It is well known that the matching polynomials of $G$ determine the signed matching polynomial of its complement \cite{Lass}. We extend this result to other graph polynomials. The main result is the duality theorem for graph polynomials.

\begin{thm} Let $D = \frac{d}{dt}$. For any simple graph $G$ and its complement $\bar G$, the following duality relations hold:
    \begin{enumerate}
        \item \( \pi_G(t) = \phi_\pi(\pi_{\bar G} ^ +(t))\) for the linear operator $\phi_\pi = e ^ {t D^2}$,
        \item \( \chi_G(t) = \phi_\chi \left[ \xi_{\bar G}(t) \right] \) where the linear operator $\phi_\chi$ is defined $\phi_\chi[t ^ n] = (t)_n = t (t - 1) \cdots (t - n + 1)$, with $(t)_n$ denoting the falling factorial,
        \item \(\xi_G(t) = \phi_\xi \left[ \chi_{\bar G}(t) \right]\), where the linear operator $\phi_\xi$ is defined by $\phi_\xi\left[t ^ n\right] = \sum_{k = 1} ^ n \stirlings{n}{k} t ^ k$, with $\stirlings{n}{k}$ denoting the Stirling number of the second kind.
    \end{enumerate}
\end{thm}

In Section 3, we show that the path-cover polynomials satisfy the duality relation. Duality allows us to interpret the Lebesgue integral and the inner product in the space of associated Laguerre polynomials $L_{n}^{(-1)}(t)$. In particular, they are expressed in terms of the number of Hamiltonian cycles. Godsil \cite{Godsil} and Lass \cite{Lass} found similar relations for matching polynomials in the context of Hermite polynomials. We complement Lass's results by interpreting the inner product in the space of Hermite polynomials as the number of perfect matchings in join graphs.
% In general, we show that all 4 polynomials satisfy duality relations. We use the duality of graph polynomials to show the relationship (\ref{main}) in Sections 3, 4, and 5.

Section 6 is devoted to applications. A cograph is a graph generated from a single vertex using join $(\triangledown)$ and union $(\cup)$ operations. In Section 6.1, we introduce an efficient algorithm that computes the number of perfect matchings, Hamiltonian paths, Hamiltonian cycles, and proper colorings of cographs.

Horák and Továrek \cite{Horak} obtained a recurrence formula for the number of Hamiltonian cycles for a graph $K_{\alpha_1, \alpha_2, \cdots, \alpha_m}$. In Sections 6.2 and 6.3, we find several ways to express the number of Hamiltonian paths and Hamiltonian cycles for the graph $K_{\alpha_1, \alpha_2, \cdots, \alpha_m}$. We also find an explicit formula for the complete $m\text{-partite}$ graph $K_{n, n, \cdots, n}$. Our results enable efficient computations of Hamiltonian paths and Hamiltonian cycles in complete $m\text{-partite}$ graphs.

\section{Graph Polynomials}

Let $\mathcal G$ be a set of graphs $G$ without loops and multiple edges. Let $K$ be a ring and $x$ be an indeterminate. A graph polynomial $P_G(x)$ is a polynomial in $K[X]$ such that $P_G(x) = P_H(x)$ whenever $G$ and $H$ are isomorphic.

Let $G = (V, E)$ be a simple graph with $n$ vertices. Below, we outline the graph polynomials we consider.

\textbf{The path-cover polynomial.} A $k\text{-path cover}$ of $G$ is a set of $k$ vertex-disjoint directed paths covering $V$, including trivial paths consisting of a single vertex. Note that a $1\text{-path cover}$ is a directed Hamiltonian path. Let $p_k(G)$ be the number of $k\text{-path covers}$ in $G$. The path-cover polynomial $\pi_G(t)$ is defined by

\[
\pi_G(t) = \sum_{k = 1} ^ n p_k(G) t ^ k.
\]
We define the signed path-cover polynomial $\pi_G^+(t)$ as follows:

\[
\pi_{G}^{+}(t) = \sum_{k = 1} ^ {n} (-1) ^ {n - k} p_k(G) t ^ k.
\]
In particular, there is a relationship between $\pi_G(t)$ and $\pi_G^+(t)$:

\[
\pi_G(t) = (-1) ^ {n} \pi_G^+(-t).
\]

The path-cover polynomial was originally introduced by D'Antona and Munarini \cite{DAntona}. Their definition is broader and includes directed graphs. Although we restrict our attention to undirected graphs, our results naturally extend to the case of directed graphs using techniques analogous to those developed in this paper.

\,

\textbf{The matching polynomial.} A $k\text{-matching}$ of a graph $G$ is a subset of edges $E' \subseteq E$ with cardinality $|E'| = k$ such that no two edges in $E'$ have a shared vertex. Denote $m_k(G)$ by the number of $k\text{-matchings}$ in $G$. We define the matching polynomial $\mu_G(t)$ as follows:

\[
\mu_G(t) = \sum_{k = 0} ^ {n / 2}  m_k(G)t ^ {n - 2 k}.
\]
The signed matching polynomial $\mu^+_{ G}(t)$ is defined as follows:

\[
\mu^+_{ G}(t) = \sum_{k = 0} ^ {n / 2} (-1) ^ k m_k(G) t ^ {n - 2k}.
\]
There is a relationship with the unsigned matching polynomial:

\[
\mu_G^+(t) = (-i) ^ n \mu_G(it), \qquad i = \sqrt{-1}.
\]

\,

\textbf{The clique-cover polynomial} A $k\text{-clique-cover}$ of $G$ is a set of $k$ vertex-disjoint cliques covering $V$, including trivial cliques consisting of a single vertex. Let $s_k(G)$ be the number of $k\text{-clique-covers}$ in $G$. The clique-cover polynomial of $G$ is defined as follows:

\[
\xi_G(t) = \sum_{k = 1} ^ n s_k(G) t ^ k.
\]

\textbf{The chromatic polynomial.} Let $\chi_G(t)$ denote the number of proper colorings of a graph $G$ in $t$ colors. Birkhoff \cite{Birkhoff} proved that $\chi_G(t)$ is a polynomial. We call $\chi_G(t)$ the \textit{chromatic polynomial} of $G$. 

\,

The union of graphs $G_1 = (V_1, E_1)$ and $G_2 = (V_2, E_2)$ with disjoint vertex sets $V_1$ and $V_2$ and edge sets $E_1, E_2$ is the graph with vertex set $V_1 \cup V_2$ and edge set $E_1 \cup E_2$. We denote the union of graphs as $G_1 \cup G_2$. The join of graphs $G_1 \triangledown G_2$ is the graph union together with all the edges between $V_1$ and $V_2$. A join graph is a graph that can be expressed as the join of two graphs.

Denote $t^{\times n} = \overbrace{t  \times \ldots \times t}^{n \text{ times}}$ for a commutative and associative multiplication $(\times)$. We use the auxiliary lemma below to compute graph polynomials of join graphs.

\begin{lm}\label{iso_prod}
    Let $(\mathbb K[t], \times)$ be a polynomial algebra with a commutative and associative multiplication and a basis $\{t^ {\times n}\}_{n\geq 0}$. Then, there exists an operator $\phi$ such that the following holds for any $a(t), b(t) \in \mathbb K[t]$.

    \[
    a(t) \times b(t) = \phi\left[ \phi^{-1}[a(t)] \phi^{-1}[b(t)]\right].
    \]
\end{lm}

\begin{proof}
    Because $\{t^ {\times n}\}_{n\geq 0}$ forms a basis of $\mathbb K[t]$, there are two ways to express a polynomial $p(t)$:

    \[
    p(t) = \sum_{i \geq 0} a_i t ^ i = \sum_{i \geq 0} a_i' t ^ {\times i}.
    \]
    Let $\phi$ be an operator that changes the basis as follows:

    \[
    \phi[t ^ n] = t ^ {\times n}.
    \]
    Then, one can verify that
    
    \begin{align*}
    a(t) \times b(t)
        &= \sum_{i \geq 0} \sum_{j \geq 0} a'_i b'_j t ^ {\times (i + j)} \\
        &= \phi \left[\sum_{i \geq 0} \sum_{j \geq 0} a_i' b_j' t ^ {(i + j)} \right] \\
        &= \phi \left[ \sum_{i \geq 0} a_i' t ^ i \sum_{j \geq 0} b_j' t ^ j\right] \\
        &= \phi \left[ \phi ^ {-1} \left[ \sum_{i \geq 0} a_i' t ^ {\times i} \right] \phi ^ {-1} \left[ \sum_{j \geq 0} b_j' t ^ {\times j} \right] \right] \\
        &= \phi\left[ \phi^{-1}[a(t)] \phi^{-1}[b(t)]\right].
    \end{align*}
\end{proof}

\section{The path-cover polynomial}

In this section, we outline the properties of the path-cover polynomial. In Section 3.1, we express the path-cover polynomials of join graphs. In Section 3.2, we prove the duality theorem for the path-cover polynomial. The second proof of Theorem \ref{path-main} uses duality. In Sections 3.3 and 3.4, we show the relationships between the associated Laguerre polynomials and the path-cover polynomial.

While this paper focuses on simple graphs, analogous approaches can extend our results to path-cover polynomials on directed graphs. This requires using the following analogues for graph operations on directed graphs. If $G$ and $H$ are directed graphs, the join graph $G \triangledown H$ is the graph union of $G$ and $H$ together with all bidirectional edges between vertices of $G$ and vertices of $H$. The graph complement $\bar G$ of a directed graph $G$ is a graph where a directed edge exists if and only if it does not exist in $G$, with the restriction that self-loops are not allowed.

% If $G$ and $H$ are directed graphs, define the join graph $G \triangledown H$ as the graph union of $G$ and $H$ together with all bidirectional edges between vertices of $G$ and vertices of $H$. Also, define the graph complement $\bar G$ of a directed graph $G$ as a graph where an directed edge exists if and only if it does not exist in $G$. While the paper focuses on simple graphs, analogous approaches can be used to extend the results to directed graphs in the context of path-cover polynomials.

\subsection{Path-cover polynomial of join graphs}

We prove the following lemma, which expresses the path-cover polynomial of a complete multipartite graph.

\begin{lm}\label{path-pre}
    Define a distributive product $(\times)$ on the set of all polynomials in $t$ by $t ^ a \times t ^ b = \pi_{K_{a, b}}(t)$. Then, this product is associative and commutative, and satisfies the following:

    \[
    t ^ {\alpha_1} \times t ^ {\alpha_2} \times \cdots t ^ {\alpha_n} = \phi_\pi \left[ \phi_\pi ^ {-1}[t ^ {\alpha_1}] \phi_\pi ^ {-1}[t ^ {\alpha_2}] \cdots \phi_\pi ^ {-1}[t ^ {\alpha_n}]\right] = \pi_{K_{\alpha_1, \alpha_2, \cdots, \alpha_n}}(t),
    \]
    where $\phi_\pi = e ^ {tD ^ 2}$, $\phi_\pi ^{-1} = e ^ {-tD ^ 2}$.
\end{lm}

\begin{proof}
    Commutativity follows from the fact that $K_{a, b} \cong K_{b, a}$.

    \[
    t ^ {a} \times t ^ b = \pi_{K_{a, b}}(t) = \pi_{K_{b , a}}(t) = t ^ b \times t ^ a.
    \]

    Now, we show the following equality:
    \begin{equation}\label{pi_fundamental}
    \pi_{K_{\alpha_1, \alpha_2, \cdots, \alpha_n}} \times t ^ {\alpha_{n + 1}} = \pi_{K_{\alpha_1, \alpha_2, \cdots, \alpha_n + 1}}(t).
    \end{equation}
    
    %% "I recommend first stating that it equals $\pi_{K_{\alpha_1, \dots, \alpha_n+1}}(t)$, and then explaining below the equation why this equality holds. "
    
    Indeed, expansion yields
    \begin{align*}
    \pi_{K_{\alpha_1, \alpha_2, \cdots, \alpha_n}}(t) \times t ^ {\alpha_{n + 1}} &= \left(\sum_{i =0} ^ {\alpha_1 + \cdots + \alpha_n} p_i(K_{\alpha_1, \alpha_2, \cdots, \alpha_n}) t ^ i \right) \times t ^ {\alpha_{n + 1}} \\
    &=\sum_{i =0} ^ {\alpha_1 + \cdots + \alpha_n} p_i (K_{\alpha_1, \alpha_2, \cdots, \alpha_n}) \pi_{K_{i, \alpha_{n + 1}}}(t),
    \end{align*}
    which can equivalently be phrased as choosing $i$ directed paths in $K_{\alpha_1, \alpha_2, \cdots, \alpha_n}$ and attaching $\alpha_{n +1 }$ points to their ends. This is the same as the path-cover polynomial of $K_{\alpha_1, \alpha_2, \cdots, \alpha_{n + 1}}$. 
    
    As a consequence of (\ref{pi_fundamental}), we obtain associativity.
    \[
    (t ^ a \times t ^ b) \times t ^ c = \pi_{K_{a, b, c}}(t) = \pi_{K_{b, c, a}}(t) = (t ^ b \times t ^ c) \times t ^ a = t ^ a \times (t ^ b \times t ^ c).
    \]

    % "Replace "Use (1) and associativity..." with "By (1) it follows...", as associativity is unnecessary for this step." 
    % It is important because we removed the brackets.
    Using (\ref{pi_fundamental}) and associativity yields
    \begin{align*}
        t ^ {\alpha_1} \times t ^ {\alpha_2} \times \cdots \times t ^ {\alpha_n} = \pi_{K_{\alpha_1, \alpha_2, \cdots, \alpha_n}}(t).
    \end{align*}

    Furthermore, $t ^ {\times n} = t \times t \times \cdots \times t = \pi_{K_{1, 1, \cdots, 1}}(t) = \pi_{K_n}(t)$ forms a basis of the polynomial space. Indeed, because $p_n(K_n) = 1$, the degree of the polynomials is $n$, and the polynomials are linearly independent. Similarly, they span the whole space.
    
    By Lemma \ref{iso_prod}, there exists an operator $\phi_\pi$ such that the following holds:
    \begin{align*}
        t ^ {\alpha_1} \times t ^ {\alpha_2} \times \cdots \times t ^ {\alpha_n} = \phi_\pi \left[ \phi_\pi^{-1} [t ^ {\alpha_1}] \phi_\pi^{-1} [t ^ {\alpha_2}] \cdots \phi_\pi^{-1} [t ^ {\alpha_n}]\right].
    \end{align*}
    
    In particular
    \begin{align*}
        \phi_\pi[t ^ n]  = t \times t \times \cdots \times t = \pi_{K_n}(t).
    \end{align*}
    
    Note that $(tD^2)$ represents the action of taking two paths and merging them into a single path. Applying the action sequentially $k$ times on $t ^ n$ results in the number of ordered $(n - k)\text{-path covers}$ of $K_n$. It is not hard to see that applying the action $k \geq n$ times yields zero. Thus, if we divide the result of the action by $k!$ to remove the order of paths in path-covers, for $k < n$, we obtain

    \[
    \frac{(t D ^ 2) ^ k}{k!} t ^ n = p_{n - k}(G) t ^ {n - k}.
    \]

    Hence, the sum evaluates as follows:

    \[
    e ^ {tD ^ 2} [t ^ n] = \sum_{k \geq 0} \frac{(t D ^ 2) ^ k}{k!} t ^ n = \sum_{k = 0} ^ {n} p_{n - k}(G) t ^ {n - k} = \pi_{K_n}(t) = \phi_\pi [t ^ n].
    \]

    Thus, we conclude that $\phi_\pi = e ^ {t D ^ 2}$.

\end{proof}

Now, we generalize Lemma \ref{path-pre} for any join graph.

\begin{thm} \label{path-main}
    Let $G$ and $H$ be simple graphs. Then, the path-cover polynomial $\pi_{G \triangledown H}(t)$ satisfies

    \[
    \pi_{G \triangledown H}(t) = \phi_\pi \left[ \phi_\pi ^ {-1}[\pi_G(t)] \phi_\pi ^ {-1}[\pi_H(t)]\right].
    \]
\end{thm}

\begin{proof}
    Let $n_1$ and $n_2$ be the number of vertices in $G$ and $H$, respectively. By definition, the coefficient of $t ^ k$ in $\pi_G(t)$ is the number of $k\text{-path covers}$. Let $\mathcal{C}$ be the set of edges of an arbitrary path-cover of $G \triangledown H$. Denote $\mathcal{P}_G$ and $\mathcal{P}_H$ as edges of $\mathcal{C}$ lying entirely in $G$ and $H$ respectively. One can see that $\mathcal{P}_G$ and $\mathcal{P}_H$ are path-covers of $G$ and $H$ since any subgraph of a path-cover is a path-cover. Let $\mathcal{A}$ and $\mathcal{B}$ denote the set of paths of a path-cover formed by edges $\mathcal{P}_G$ and $\mathcal{P}_H$. If one maps elements of $\mathcal{A}$ and $\mathcal{B}$ to vertices of $K_{|\mathcal{A}|, |\mathcal{B}|}$, then the edges $\mathcal{C} \setminus (\mathcal{P}_G \cup \mathcal{P}_H)$ naturally induce a path-cover in $K_{|\mathcal{A}|, |\mathcal{B}|}$.  Thus, the path-cover polynomial of $G \triangledown H$ can be written as follows:

    \[
    \pi_{G \triangledown H}(t) = \sum_{l_1 = 0} ^ {n_1} \sum_{l_2 = 0} ^ {n_2} p_{l_1} (G) p_{l_2}(H) \pi_{K_{l_1, l_2}}(t).
    \]
    By Lemma \ref{path-pre}, one has that $\pi_{K_{l_1, l_2}}(t) = t ^ {l_1} \times t^ {l_2} = \phi_\pi \left[ \phi_\pi ^ {-1} [t ^ {l_1}] \phi_\pi ^ {-1} [t ^ {l_2}]\right]$. Thus

    \begin{align*}
        \pi_{G \triangledown H}(t) &= \sum_{l_1 = 0} ^ {n_1} \sum_{l_2 = 0} ^ {n_2} p_{l_1} (G) p_{l_2}(H) \phi_\pi \left[ \phi_\pi ^ {-1} [t ^ {l_1}] \phi_\pi ^ {-1} [t ^ {l_2}]\right]\\
        &= \phi_\pi \left[ \sum_{l_1 = 0} ^ {n_1} \sum_{l_2 = 0} ^ {n_2} p_{l_1} (G) p_{l_2}(H)  \phi_\pi ^ {-1} [t ^ {l_1}] \phi_\pi ^ {-1} [t ^ {l_2}]
        \right]\\
        &= \phi_\pi \left[
        \phi^{-1}_\pi\left[\sum_{l_1 = 0} ^ {n_1} p_{l_1} (G)  t ^ {l_1} \right]
        \phi^{-1}_\pi\left[\sum_{l_2 = 0} ^ {n_2} p_{l_2}(H) t ^ {l_2} \right]
        \right]\\
        &= \phi_\pi \left[ \phi_\pi ^ {-1}[\pi_G(t)] \phi_\pi ^ {-1}[\pi_H(t)]\right].
    \end{align*}
\end{proof}

Because the linear operator $\phi_{\pi}$ is determined entirely from the basis elements $t^n$, solving the path-cover polynomial for the complete graph allows us to solve it for any join graph. In other words, one can fully characterize the operator $\phi_\pi$ by computing $\pi_{K_n}(t)$:

%The reason why $\phi_\pi = e ^ {tD^2}$ is because $\phi_\pi \left[ t^n\right] = \pi_{K_n}(t)$, which is true as shown in the proof of Lemma \ref{path-pre}. On a basic level, solving the path-cover polynomial for the complete graph allows us to solve it for any join graph. Thus, one can also rewrite the operator $\phi_\pi$ using Lah numbers.

\[
    \phi_\pi \left[t ^ n \right] = \pi_{K_n}(t) = \sum_{k = 1} ^ n L(n, k) t ^ k = \sum_{k = 1} ^ n \frac{n!}{k!} \binom{n - 1}{k - 1} t ^ k,
\]
\[
    \phi_{\pi} ^ {-1} \left[t ^ n \right] = \sum_{k = 1} ^ n (-1) ^ {n -k} L(n, k) t ^ k = \sum_{k = 1} ^ {n} (-1) ^ {n - k} \frac{n!}{k!} \binom{n - 1}{k - 1} t ^k,
\]
where $L(n, k)$ is the Lah number. They are the number of $k\text{-path covers}$ in $K_n.$ A well-known fact is that $L(n, k) = \frac{n!}{k!} \binom{n - 1}{k - 1}$.

\subsection{Duality of path-cover polynomials} 

We show that the path-cover polynomial of $G$ can be expressed using the signed path-cover polynomial of its complement $\bar G$.

\begin{thm}[Duality theorem]\label{path_duality} Let $G$ be a simple graph and $\bar G$ be its complement. Then the following duality relation for path-cover polynomials holds:
    \[
    \pi_G(t) = \phi_\pi\left[\pi_{\bar G}^+(t) \right].
    \]
\end{thm}

\begin{proof}
    % Equation (2) requires further explanation.
    First, we express the number $k\text{-path covers}$ $p_K(G)$ using the inclusion-exclusion principle as follows:
    \begin{equation}\label{path_inclusion-exclusion}
        p_k(G) = \sum_{j = 0} ^ {n - 1} (-1) ^ j p_k(K_{n - j}) p_{n - j}(\bar G).
    \end{equation}
    Indeed, because $p_{n - j}(\bar G)$ is the number of $(n-j)\text{-path}$ covers with $j$ edges that are not in $G$, and $p_{k}(K_{n - j})$ is the number of ways to merge them, the number $p_k(K_{n - j}) p_{n - j}(\bar G)$ is the number of ways to choose $k\text{-path covers}$ such that at least $j$ edges are not in $G$. By the inclusion-exclusion principle, we obtain the number of $k\text{-path}$ covers that lie entirely in $G$.

    Expanding $\phi_\pi\left[\pi_{\bar G}^+(t) \right]$ yields the following:

    \begin{align*}
        \phi_\pi\left[\pi_{\bar G}^+(t)\right] &= \sum_{j = 1} ^ n (-1) ^ {n - j} p_j(\bar G) \phi_\pi \left[ t ^ j\right] \\ &=
        \sum_{j = 1} ^ n (-1) ^ {n - j} p_j(\bar G)  \sum_{k = 1} ^ {j} p_k(K_j) t ^ k \\ &=
        \sum_{k = 1} ^ n t ^ k \sum_{j = 1} ^ n p_k(K_j) p_j(\bar G) (-1) ^ {n - j} \\ &= 
        \sum_{k = 1} ^ n t ^ k \sum_{j = 0} ^ {n - 1} (-1) ^ j p_k(K_{n - j}) p_{n - j}(\bar G).
    \end{align*}
    By (\ref{path_inclusion-exclusion}), one obtains
    \[
    \phi_\pi\left[\pi_{\bar G}^+(t)\right] = \sum_{k = 1} ^ n t ^ kp_k(G) = \pi_G(t).
    \]
\end{proof}

With the duality theorem, we are now able to give a simpler proof of Theorem \ref{path-main}.

\begin{proof}[The second proof of Theorem \ref{path-main}]

\begin{align*}
\pi_{G \triangledown H}(t) &= \pi_{\overline{\bar G \cup \bar H}}(t) = \phi_\pi\left[\pi_{\bar G \cup \bar H} ^ {+}(t)\right]\\
&= \phi_\pi\left[\pi^+_{\bar G}(t) \pi^+_{\bar H}(t)\right] = \phi_\pi \left[ \phi_\pi ^ {-1}\left[ \pi_G(t) \right] \phi_\pi ^ {-1}\left[ \pi_H(t) \right] \right].
\end{align*}
\end{proof}

\subsection{The associated Laguerre polynomials}

Let the measure $\mu$ be such that \(d\mu(t) = t ^ {-1} e ^ {-t} dt\) on $(0, \infty)$. Throughout this section, we assume that all polynomials have zero constant term. For polynomials $f(t)$, define $\mathcal{L}[f(t)]$ as the Lebesgue integral with respect to the measure $\mu$.

\begin{align*}
    \mathcal L[f(t)] = \int_{0} ^ {\infty} f(t) d\mu(t)
    = \int_0 ^ {\infty} t ^ {-1} e ^ {-t} f(t) dt.
\end{align*}

Define the inner product induced by the measure $\mu$ for polynomials $f(t)$ and $g(t)$.

\[
\langle f, g \rangle = \mathcal L \left[ f(t)g(t) \right]
= \int_0^{\infty} t ^ {-1} e ^ {-t} f(t) g(t) dt.
\]

Then, the associated Laguerre polynomials $L_n^{(-1)}(t)$ is the following:

\[
L_n^{(-1)}(t) = \frac{(-1)^n}{n!} e ^ {-t D^2} \left[ t ^ n\right].
\]

It is well-known that the associated Laguerre polynomials $L_n^{(-1)}(t)$ are orthogonal under the inner product induced by $\mu$. The associated Laguerre polynomials can be expressed as the signed path-cover polynomials of complete graphs.

\[
(-1) ^ n n! L_n^{(-1)}(t) = e^{-tD^2}\left[t ^ n\right]=
\pi_{K_n}^+(t).
\]

Therefore, signed path cover polynomials of complete graphs are orthogonal. Moreover, signed path-cover polynomials are a generalization of the associated Laguerre polynomials $L_n^{(-1)}(t)$. 

Denote $c(G)$ by the number of directed Hamiltonian cycles in $G$. We show that the Lebesgue integral can be expressed in terms of the number of directed Hamiltonian cycles $c(G)$ and $c(\bar G)$.

\begin{thm}\label{cycle_characteristic} Let $G$ be a simple graph. The Lebesgue integral of the path-cover polynomial with respect to the measure $\mu$ satisfies

\[
\mathcal L \left[ \pi_{\bar G} ^ +(t)\right] = c(G) + (-1) ^ {n - 1} c(\bar G).
\]

\end{thm}

\begin{proof}
    By the definition of $\pi^+_{\bar G}(t)$
    \[\left. \pi^+_{\bar G}(t) \right|_{t ^ k \rightarrow (k - 1)!} = \sum_{k = 0} ^ {n - 1} (-1) ^ k ( n -k - 1)! p_{n - k}(\bar G).
    \]

    The number of directed Hamiltonian cycles in $G$ can be expressed using the inclusion-exclusion principle as follows:

    \[
    c(G) = \sum_{k = 0} ^ {n - 1} (-1) ^ k c(K_{n - k}) p_{n - k} (\bar G) + (-1) ^ n c(\bar G).
    \]
    Indeed, since $p_{n - k} (\bar G)$ is the number of path covers with $k$ edges that are not in $G$, and $c(K_{n - k})$ is the number of ways to merge them into a directed cycle, the number $c(K_{n - k}) p_{n - k} (\bar G)$ is the number of cycles with at least $k$ edges outside of $G$. The final term, $(-1)^n c(\bar G)$, accounts for the case representing the cycles that lie entirely outside of $G$. Thus, by the inclusion-exclusion principle, the sum evaluates to $c(G)$. 

    Because $c(K_{n - k}) = (n - k -1)!$, one obtains

    \[
    c(G) = \left. \pi^+_{\bar G}(t) \right|_{t ^ k \rightarrow (k - 1)!} + (-1) ^ n c(\bar G).
    \]
    Rearranging the terms yields
    \[
    \left. \pi^+_{\bar G}(t) \right|_{t ^ k \rightarrow (k - 1)!} = c(G) + (-1) ^ {n - 1} c(\bar G).
    \]
    By $\pi^+_{\bar G}(t) = (-1) ^ n \pi_{\bar{G}}(-t)$, we get
    \[
    \left. \pi_{\bar G}(t) \right|_{t ^ k \rightarrow (-1) ^ {k - 1}(k - 1)!} = (-1) ^ {n - 1} c(G) + c(\bar G).
    \]
    This is the same as 
    \[
    \left. \pi_{G}(t) \right|_{t ^ k \rightarrow (-1) ^ {k - 1}(k - 1)!} = c(G) + (-1) ^ {n - 1} c(\bar G).
    \]

    Note that the Lebesgue integral of $t ^ n$ is the following:

    \[
    \mathcal{L}[t ^ n] = \int_{0} ^ {\infty} e ^ {-t} t ^ {n - 1} dt = (n - 1)!.
    \]
    
    Therefore, one obtains the desired equality:
    \[
    \mathcal L \left[ \pi_{\bar G} ^ +(t)\right] = \int_0 ^ {\infty} t ^ {-1} e ^ {-t} \pi_{\bar G} ^ +(t) dt = c(G) + (-1) ^ {n - 1} c(\bar G).
    \]
\end{proof}

For a simple graph $G$, if one assumes its complement $\bar G$ is not Hamiltonian, it is possible to express the number of Hamiltonian cycles in $G$.

\begin{crl}\label{hc_char_0}
    If $\bar G$ is not Hamiltonian, then

    \[
    c(G) = \mathcal L \left[ \pi^+_{\bar G}(t)\right] =\pi_G(t) |_{t ^ k \rightarrow (-1) ^ {k - 1} (k - 1)!}.
    \]
\end{crl}

\textbf{Example.} Let $G$ be the Petersen graph. The signed path-cover polynomial is 
\(
\pi^+_G(t) = -240 t + 3120 t^2 - 11160 t^3 + 18280 t^4 - 15912 t^5 + 7860 t^6 - 
 2240 t^7 + 360 t^8 - 30 t^9 + t^{10}.
\) The Petersen graph is not Hamiltonian, which means that applying Corollary \ref{hc_char_0} gives that there are $6432$ directed Hamiltonian cycles in $\bar G$. The complement of the Petersen graph is also the line graph of the complete graph with five vertices $L(K_5).$ In general, evaluating the path-cover polynomial is computationally at least as hard as determining the number of directed Hamiltonian cycles.

\begin{figure}[H]
\centering
\begin{tikzpicture}[scale=1.2,
  every node/.style={circle, fill=black, inner sep=1.2pt}
]

% --- outer ring (6--10) ---
\node (6)  at (90:2) {};
\node (7)  at (18:2) {};
\node (8)  at (306:2) {};
\node (9)  at (234:2) {};
\node (10) at (162:2) {};

% --- inner ring (1--5) ---
\node (1) at (90:1) {};
\node (2) at (18:1) {};
\node (3) at (306:1) {};
\node (4) at (234:1) {};
\node (5) at (162:1) {};

% --- edges ---
\draw
(1)--(2) (1)--(5) (1)--(7) (1)--(8) (1)--(9) (1)--(10)
(2)--(3) (2)--(6) (2)--(8) (2)--(9) (2)--(10)
(3)--(4) (3)--(6) (3)--(7) (3)--(9) (3)--(10)
(4)--(5) (4)--(6) (4)--(7) (4)--(8) (4)--(10)
(5)--(6) (5)--(7) (5)--(8) (5)--(9)
(6)--(8) (6)--(9)
(7)--(9) (7)--(10)
(8)--(10);

\end{tikzpicture}
\caption{Complement of the Petersen graph $L(K_5)$.}
\label{fig1}
\end{figure}

The Paley graph of order $q$, with $q$ a prime power and $q \equiv 1 (\text{mod $4$})$, is a simple graph with $q$ vertices in which two vertices are adjacent if their difference is a quadratic residue in a finite field $\mathbb{F}_q$. Paley graphs are self-complementary, which means that Theorem \ref{cycle_characteristic} gives twice the number of directed Hamiltonian cycles in a Paley graph. A directed Hamiltonian cycle in a Paley graph can be interpreted as a circular permutation of $\{ 1, 2, \cdots, q\}$ in which the difference of two adjacent numbers is a quadratic residue. 

Let $G$ be a Paley graph of order $9$. Then, the signed path-cover polynomial $\pi^+_G(t) = 1512 t - 11736 t^2 + 26952 t^3 - 26640 t^4 + 13248 t^5 - 3528 t^6 + 504 t^7 - 36 t^8 + t^9$. Using Theorem \ref{cycle_characteristic}, one obtains that there are $96$ directed Hamiltonian cycles.

\begin{figure}[H]
    \centering
    \begin{tikzpicture}[scale=1,
      every node/.style={circle, fill=black, inner sep=1.2pt}
    ]
        % Outer cycle positions (rough circular order)
        \node (3) at (90:2) {};
        \node (4) at (30:2) {};
        \node (5) at (-30:2) {};
        \node (2) at (-90:2) {};
        \node (7) at (-150:2) {};
        \node (9) at (150:2) {};
    
        % Inner triangle
        \node (6) at (90:0.8) {};
        \node (1) at (210:0.8) {};
        \node (8) at (-30:0.8) {};
    
        % Edges
        \draw (1) -- (2);
        \draw (1) -- (3);
        \draw (1) -- (4);
        \draw (1) -- (7);
    
        \draw (2) -- (5);
        \draw (2) -- (7);
        \draw (2) -- (8);
    
        \draw (3) -- (4);
        \draw (3) -- (8);
        \draw (3) -- (9);
    
        \draw (4) -- (5);
        \draw (4) -- (6);
    
        \draw (5) -- (6);
        \draw (5) -- (8);
    
        \draw (6) -- (7);
        \draw (6) -- (9);
    
        \draw (7) -- (9);
    
        \draw (8) -- (9);
    \end{tikzpicture}
    \caption{The Paley graph of order $9$.}
\end{figure}

It is also possible to compute $c(G \triangledown H)$ using Corollary \ref{hc_char_0} because the complement of a join graph is disconnected. This gives a nice interpretation for the inner product.

\begin{thm}\label{laguerre_inner} For simple graphs $G$ and $H$, the following holds:

\[
\langle \pi_{\bar G} ^ +, \pi_{\bar H} ^ + \rangle = c(G\triangledown H).
\]
\end{thm}

\begin{proof}
    By the definition of the inner product
    \begin{align*}
        \langle \pi^{+}_{\bar G}, \pi^{+}_{\bar H} \rangle &= \int_0 ^ \infty x^{-1} e ^ {-x} (\pi^{+}_{\bar G}(x)\pi^{+}_{\bar H}(x)) dx \\
        &= \int_0 ^ \infty x^{-1} e ^ {-x} (\pi^{+}_{\bar G \cup \bar H}(x)) dx = \int_0 ^ \infty x ^ {-1} e ^ {-x} (\pi^{+}_{\overline{ G \triangledown H}}(x)) dx.
    \end{align*}
    By Theorem \ref{cycle_characteristic}, one obtains
    \[
    \langle \pi^{+}_{\bar G}, \pi^{+}_{\bar H} \rangle = c(G \triangledown H).
    \]
\end{proof}

A consequence of this result is that the square of the norm of $\pi_{\bar G} ^ +(t)$ is the number of directed Hamiltonian cycles in $G \triangledown G$.

\begin{crl} Define the norm induced by the inner product by

\[
\left\Vert f \right\Vert = \sqrt{\langle f, f \rangle}.
\]

Then, the square of the norm of a polynomial $\pi_{\bar G} ^ +(t)$ equals the number of directed Hamiltonian cycles of $G \triangledown G$:

\[
\left\Vert \pi_{\bar G} ^ + \right\Vert ^ 2 = c(G \triangledown G).
\]
\end{crl}

This leads to the Cauchy-Schwarz inequality in the number of directed Hamiltonian cycles of join graphs.

\begin{crl} Let $G$ and $H$ be simple graphs. The Cauchy-Schwarz inequality for the number of directed Hamiltonian cycles is given by the following:

    \[
    c(G \triangledown H) ^ 2 \leq c(G \triangledown G) c(H \triangledown H).
    \]
    
\end{crl}

Another way to compute the number of directed Hamiltonian cycles in $G \triangledown H$ involves the unsigned polynomials $\pi_G(t)$ and $\pi_H(t)$.

\begin{prp}\label{HC}

Suppose $G$ and $H$ are simple graphs with order $n \geq 2$. Then, the number of directed Hamiltonian cycles in $G \triangledown H$ is

\[
\sum_{k \geq 1} k! (k - 1)! \left( \left[ t ^ k\right] \pi_G(t) \right) \left( \left[ t ^ k\right] \pi_H(t) \right).
\]

\end{prp}

\begin{proof}
    It is well-known that the complete bipartite graph $K_{a, b}$ has a Hamiltonian cycle iff $a = b \geq 2$, in which case there are $a!(a - 1)!$ directed Hamiltonian cycles. Suppose $\mathcal A$ is a path-cover of $G$ and $\mathcal B$ is a path-cover of $H$. If one thinks of elements of $\mathcal{A}$ and $\mathcal{B}$ as vertices of $K_{|\mathcal{A}|, |\mathcal{B}|}$, then all Hamiltonian cycles of $K_{|\mathcal{A}|, |\mathcal{B}|}$ define Hamiltonian cycles of $G \triangledown H$. Conversely, Hamiltonian cycles of $G \triangledown H$ can be expressed as Hamiltonian cycles in $K_{|\mathcal{A}|, |\mathcal{B}|}$ for some path-covers $\mathcal{A}$ in $G$ and $\mathcal{B}$ in $H$. Thus, the number of directed Hamiltonian cycles in $G \triangledown H$ is
    
    % Then, forming a Hamiltonian cycle in $G \triangledown H$ is equivalent to alternatively joining paths from $\mathcal A$ and from $\mathcal B$ so that they form a Hamiltonian cycle. If one thinks of elements in $\mathcal{A}$ and $\mathcal{B}$ as vertices in $K_{|\mathcal A|, |\mathcal B|}$, then it is evident that the number of Hamiltonian cycles in $G \triangledown H$:

    \[
    \sum_{k \geq 1} k! (k - 1)! \left( \left[ t ^ k\right] \pi_G(t) \right) \left( \left[ t ^ k\right] \pi_H(t) \right).
    \]
\end{proof}

We can also express the number of directed Hamiltonian paths in terms of an integral.

\begin{prp}\label{ham_integral}
    Let $G$ be a simple graph. The number of directed Hamiltonian paths can be expressed as follows:

    \[
    \langle \pi_{\bar G}^+, t\rangle = \int_{0} ^ {\infty} e^{-t} \pi_{\bar G}^+(t) dt.
    \]
\end{prp}

\begin{proof}
    Directed Hamiltonian cycles in $G \triangledown K_1$ correspond uniquely to directed Hamiltonian paths in $G$. In a Hamiltonian cycle, the neighbors of a vertex corresponding to $K_1$ can be thought of as the two ends of the Hamiltonian path in $G$. Use Theorem \ref{laguerre_inner} to find the number of Hamiltonian cycles in $G \triangledown K_1$.
\end{proof}

\section{The matching polynomial}

This section outlines the properties of the matching polynomial. The proof strategies are similar to those of the path-cover polynomial. Section 4.1 expresses the matching polynomial of join graphs. Section 4.2 discusses the duality theorem, from which follows the second proof of Theorem \ref{matching_main}. Finally, we find the number of perfect matchings of a join graph using the inner product of Hermite polynomials.

\subsection{Matching polynomial of join graphs}

First, we prove the following lemma, which expresses the matching polynomial of the complete multipartite graph.

\begin{lm}\label{match_pre}

Define a distributive product $(\times)$ as $t ^ a \times t ^ b = \mu_{K_{a, b}}(t)$. Then, it is commutative and associative, and satisfies the following:

\[
t ^ {\alpha_1} \times t^ {\alpha_2} \times \cdots \times t ^ {\alpha_n} = \phi_\mu\left[
\phi_\mu^{-1}[t ^ {\alpha_1}] \phi_\mu^{-1}[t ^ {\alpha_2}] \cdots \phi_\mu^{-1}[t ^ {\alpha_n}]
\right] = \mu_{K_{\alpha_1, \alpha_2, \cdots, \alpha_N}}(t),
\]
where $\phi_\mu = e ^ {D^2 / 2}$, $\phi_\mu ^ {-1}= e ^ {-D^2 / 2}$.

\end{lm}

\begin{proof}
    Commutativity follows from the fact that $K_{a, b} \cong K_{b, a}$:
    \[
    t ^ {a} \times t^{b} = \mu_{K_{a, b}}(t) = \mu_{K_{b, a}}(t) = t ^ {b} \times t ^ {a}.
    \]
    Now, we show the following equality:
    \begin{equation}\label{mu_fundamental}
        \mu_{K_{\alpha_1, \alpha_2, \cdots, \alpha_{n}}} \times t ^ {\alpha_{n + 1}} = \mu_{K_{\alpha_1, \alpha_2, \cdots, \alpha_{n + 1}}}(t).
    \end{equation}
    Indeed, the expansion yields
    \begin{align*}
        \mu_{K_{\alpha_1, \alpha_2, \cdots, \alpha_n}}(t) \times t ^ {\alpha_{n + 1}} &= \left(\sum_{i = 0} ^ {\alpha_1 + \cdots + \alpha_n} m_{((\alpha_1 + \cdots +\alpha_n) - i) / 2} (K_{\alpha_1, \alpha_2, \cdots, \alpha_n}) t ^ i \right) \times t ^ {\alpha_{n + 1}} \\
        &=\sum_{i = 0} ^ {\alpha_1 + \cdots + \alpha_n} m_{((\alpha_1 + \cdots + \alpha_n) - i) / 2} (K_{\alpha_1, \alpha_2, \cdots, \alpha_n}) \mu_{K_{i, \alpha_{n + 1}}}(t).
    \end{align*}
    The expression above counts matchings in which $i$ vertices of the graph $K_{\alpha_1, \alpha_2, \cdots, \alpha_n}$ are unmatched, and these unmatched vertices are paired with $\alpha_{n + 1}$ other vertices. In other words, the expression is the same as the matching polynomial of $K_{\alpha_1, \alpha_2, \cdots, \alpha_{n + 1}}$.
    Associativity is immediate.

    \[
    (t ^ a \times t ^ b) \times t ^ c = \mu_{K_{a, b, c}}(t) = \mu_{K_{b, c, a}}(t) = (t ^ b \times t ^ c) \times t ^ a = t ^ a \times (t ^ b \times t ^ c).
    \]
    Associativity and (\ref{mu_fundamental}) yields
    \[
    t ^ {\alpha_1} \times t ^ {\alpha_2} \times \cdots \times t ^ {\alpha_n} = \mu_{K_{\alpha_{1}, \alpha_2, \cdots \alpha_n}}(t).
    \] 

    Note that $t ^ {\times n} = t \times t \times \cdots \times t = \mu_{K_{1, 1, \cdots, 1}}(t) = \mu_{K_n}(t)$ forms a basis of the polynomial space. Indeed, because the degree of polynomials $\mu_{K_n}(t)$ is $n$, they are linearly independent. Similarly, they span the whole space.
    
    By Lemma \ref{iso_prod}, there exists some operator $\phi_\mu$ such that the following holds:
    \[
    t ^ {\alpha_1} \times t^ {\alpha_2} \times \cdots \times t ^ {\alpha_n} = \phi_\mu\left[
    \phi_\mu^{-1}[t ^ {\alpha_1}] \phi_\mu^{-1}[t ^ {\alpha_2}] \cdots \phi_\mu^{-1}[t ^ {\alpha_n}]
    \right].
    \]
    Note that
    \[
    t \times t \times \cdots \times t = \mu_{K_{1, 1, \cdots, 1}}(t) = \mu_{K_n}(t).
    \]
    In other words
    \[
    \phi_\mu \left[t ^ n \right] = \mu_{K_n}(t),
    \]
    which holds when $\phi_\mu =  e ^ {D^2 / 2}$.

    \[
    e ^ {D^2 / 2}(t ^ n) = \sum_{k \geq 0} \frac{D ^ {2k}}{2 ^ k k!} t ^ n = \sum_{k \geq 0} t ^ {n - 2k} \frac{n!}{(n - 2k)! 2 ^ k k!} = \sum_{k \geq 0} t ^ {n - 2k} \binom{n}{2k}(2k - 1)!!,
    \]
    where $(2k - 1)!! = (2k - 1)(2k - 3) \cdots 3 \cdot 1$ is the double factorial, which enumerates the number of perfect matchings in a complete graph on $2k$ vertices. The last expression is indeed $\mu_{K_n}(t)$, because there are $\binom{n}{2k}(2k - 1)!!$ ways to choose $k\text{-matchings}$ in $K_n$. 
\end{proof}

Now we generalize Lemma \ref{match_pre} for any join graph.

\begin{thm}\label{matching_main}
    Let $G$ and $H$ be simple graphs. Then, the matching polynomial $\mu_{G \triangledown H}(t)$ satisfies

    \[
    \mu_{G \triangledown H}(t) = \phi_\mu \left[ \phi_\mu ^ {-1}[\mu_G(t)] \phi_\mu ^ {-1}[\mu_H(t)]\right].
    \]
\end{thm}

\begin{proof}
    Let $n_1$ and $n_2$ be the number of vertices in $G$ and $H$, respectively. By the definition of matching polynomials, the coefficient of $t^l$ in $\mu_G(t)$ is the number of matchings such that exactly $l$ vertices are unmatched. The matching polynomial \( \mu_{G \triangledown H}(t)\) can be written as follows:

    \[
    \mu_{G \triangledown H}(t) = \sum_{l_1 = 0} ^ {n_1} \sum_{l_2 = 0} ^ {n_2} m_{(n_1 - l_1) / 2} (G) m_{(n_2 - l_2) / 2}(H) \mu_{K_{l_1, l_2}}(t).
    \]
    This is true because, if $\mathcal M$ is a matching in $G \triangledown H$, and if $\mathcal{A} \subseteq \mathcal{M}$ is a matching in $G$ and $\mathcal{B} \subseteq \mathcal{M}$ is a matching in $H$, the matchings $\mathcal{M} \setminus (\mathcal{A} \cup \mathcal{B})$ are precisely the matchings of the complete bipartite graph. 
    
    % $K_{|V(G)| - 2|\mathcal{A}|, |V(H)| - 2|\mathcal{B}|}$.

    % Let $M$ be a matching of $G \triangledown H$. The edges of $M$ either lie in $G$, or lie in $H$, or connect the vertices from $G$ and $H$.  Hence

    By Lemma \ref{match_pre}, one has that $\mu_{K_{l_1, l_2}}(t) = t ^ {l_1} \times t^ {l_2} = \phi_\mu \left[ \phi_\mu ^ {-1} [t ^ {l_1}]  \phi_\mu ^ {-1} [t ^ {l_2}]\right]$. Thus

    \begin{align*}
        \mu_{G \triangledown H}(t) &= \sum_{l_1 = 0} ^ {n_1} \sum_{l_2 = 0} ^ {n_2} m_{(n_1 - l_1) / 2} (G) m_{(n_2 - l_2) / 2}(H) \phi_\mu \left[ \phi_\mu ^ {-1} [t ^ {l_1}]  \phi_\mu ^ {-1} [t ^ {l_2}]\right] \\
        &= \phi_\mu \left[ \sum_{l_1 = 0} ^ {n_1} \sum_{l_2 = 0} ^ {n_2} m_{(n_1 - l_1) / 2} (G) m_{(n_2 - l_2) / 2}(H) \phi_\mu ^ {-1} [t ^ {l_1}]  \phi_\mu ^ {-1} [t ^ {l_2}]
        \right]\\
        &= \phi_\mu \left[
        \phi^{-1}_\mu\left[\sum_{l_1 = 0} ^ {n_1} m_{(n_1 - l_1) / 2} (G)  t ^ {l_1} \right]
        \phi^{-1}_\mu\left[\sum_{l_2 = 0} ^ {n_2} m_{(n_2 - l_2) / 2} (H) t ^ {l_2} \right]
        \right] \\
        &= \phi_\mu \left[ \phi_\mu ^ {-1} \left[ \mu_G(t)\right] \phi ^ {-1}_\mu \left[ \mu_H(t)\right] \right].
    \end{align*}

    % Define an associative and distributive product $(\times)$ as follows:

    % \[
    % t ^ {l_1} \times t^{l_2} \times \cdots \times t ^ {l_N} = \mu_{K_{l_1, \cdots l_N}}(t).
    % \]

    % Note that for any permutation $\sigma$, the isomorphism $K_{l_1, \cdots l_N} \cong K_{l_{\sigma(1)}, \cdots l_{\sigma(N)}}$ holds. Then, the product $(\times)$ is commutative:

    % \[
    % t ^ {l_1} \times \cdots \times t ^ {l_N} = \mu_{K_{l_1, \cdots l_N}}(t) = \mu_{K_{l_{\sigma(1)}, \cdots l_{\sigma(N)}}}(t) = t ^ {l_{\sigma(1)}} \times \cdots \times t ^ {l_{\sigma(N)}}.
    % \]

    % Using Lemma \ref{algebra}, one has that:

    % \[
    % t ^ {l_1} \times 
    % \]
    
    % First, we show that there exists an associative, commutative, and distributive product $(\times)$ such that

    % \[
    % \mu_{G \triangledown H}(t) = \mu_G(t) \times \mu_H(t).
    % \]

    % By definition 
    % Indeed, because one can obtain \(\mu_{G \triangledown H}(t)\) by joining the unmatched vertices in $G$ and $H$, there exists a product $(\times)$ which expresses \(\mu_{G \triangledown H}(t)\) as the product of \(\mu_G(t) \) and \(\mu_H(t)\).
\end{proof}

Loosely speaking, the choice of $\phi_\mu = e ^ {D ^ 2 / 2}$ is reflected by the fact that the operator selects two unmatched vertices and matches them, hence the second-order differentiation. The division by two is necessary to account for the order of differentiation.

\subsection{Duality of the matching polynomial}

We prove the duality relation between the matching polynomial of a graph $G$ and the signed matching polynomial of its complement $\bar G$ using the inclusion-exclusion principle, a result originally proven by Godsil \cite{Godsil}. Refer to the work of Lass \cite{Lass} for the proof using algebra. 

\begin{thm}[Duality theorem]\label{matching-duality} Let $G$ be a simple graph and $\bar G$ is its complement. Then the following duality for matching polynomials holds:

\[
\mu_G(t)= \phi_\mu \left[ \mu_{\bar G} ^ +(t)\right].
\]
    
\end{thm}
\begin{proof}

First, we express the number of $(n - k) / 2\text{-matchings}$ $m_{(n - k) / 2}(G)$ as follows:

\begin{equation}\label{match_inclusion-exclusion}
    m_{(n - k) / 2}(G) = \sum_{j = k} ^ n (-1) ^ {(n - j) / 2} m_{(n - j) / 2} (\bar{G})  m_{(j - k) / 2} (K_j).
\end{equation}

Indeed, because $m_{(n - j) / 2} (\bar{G})$ is the number of matchings with $n - j$ edges not in $G$, the number $m_{(n - j) / 2} (\bar{G})  m_{(j - k) / 2}$ is the number of all $(n - k) / 2\text{-matchings}$ with at least $(n - j)$ invalid edges. Therefore, by the inclusion-exclusion principle, we obtain the number $m_{(n - k) / 2}(G)$. 

Expanding the right-hand side yields

\begin{align*}
\phi_\mu \left[ \mu_{\bar{G}}^+(t)\right] &= \phi_\mu \left[ 
\sum_{i = 0} ^ {\lfloor n / 2 \rfloor} (-1) ^ i m_i(\bar{G}) t ^ {n - 2i}
\right] \\
&= \sum_{j = 0} ^ n (-1) ^ {(n - j) / 2} m_{(n - j) / 2}(\bar{G})\phi_\mu \left[ t ^ j \right] \\ 
&= \sum_{j = 0} ^ {n} (-1) ^ {(n - j) / 2} m_{(n - j) / 2}(\bar{G}) \sum_{k = 0} ^ {j} t ^ k m_{(j - k) / 2} (K_j) \\
&= \sum_{k = 0} ^ n t ^ k \sum_{j = k} ^ n (-1) ^ {(n - j) / 2} m_{(n - j) / 2} (\bar{G})  m_{(j - k) / 2} (K_j).
\end{align*}

By (\ref{match_inclusion-exclusion}), one obtains

$$ \phi_\mu \left[ \mu_{\bar{G}} ^ +(t)\right] = \sum_{k = 0} ^ n t ^ k m_{(n - k) / 2}(G) = \mu_G(t). $$

\end{proof}

With the duality theorem of the matching polynomials, we can give the second proof of Theorem \ref{matching_main}.

\begin{proof}[The second proof of Theorem \ref{matching_main}]
    \begin{align*}
        \mu_{G \triangledown H}(t) &= \mu_{\overline{\bar G \cup \bar H}}(t) = \phi_\mu \left[\mu_{\bar G \cup \bar H} ^ {+}(t) \right] \\
        &= \phi_\mu\left[\mu^+_{\bar G}(t) \mu^+_{\bar H}(t)\right] = \phi_\mu \left[ \phi_\mu ^ {-1}\left[ \mu_G(t) \right] \phi_\mu ^ {-1} \left[ \mu_H(t) \right] \right].
    \end{align*}
\end{proof}

Hermite polynomials are orthogonal under the following inner product:

\[
\langle f, g\rangle = \frac{1}{\sqrt{2 \pi}}\int_{-\infty} ^ {\infty} e ^ {-t ^ 2 / 2} f(t) g(t) dt.
\]

A slight addition to the results of Lass \cite{Lass} is the interpretation of the inner product of Hermite polynomials as the number of perfect matchings in a join graph. Denote $m(G)$ as the number of perfect matchings in $G$.

\begin{thm}
    Let $G$ and $H$ be simple graphs. Then, the number of perfect matchings of $G \triangledown H$ can be expressed as the inner product of matching polynomials as follows:

    \[
    m(G\triangledown H) = \langle \mu_{\bar G} ^ +,\mu_{\bar H}^+\rangle = \frac{1}{\sqrt{2\pi}} \int_{-\infty} ^ {\infty} e ^ {-t ^ 2 / 2} \mu^+_{\bar G}(t) \mu_{\bar H}^+(t)dt.
    \]
\end{thm}

\begin{proof}
    By the definition of the inner product

    \begin{align*}
        \langle \mu_{\bar G} ^ +,\mu_{\bar H}^+\rangle &= \frac{1}{\sqrt{2\pi}} \int_{-\infty} ^ {\infty} e ^ {-t ^ 2 / 2} \mu^+_{\bar G}(t) \mu_{\bar H}^+(t)dt \\&= \frac{1}{\sqrt{2\pi}} \int_{-\infty} ^ {\infty} e ^ {-t ^ 2 / 2} \mu^+_{\bar G \cup \bar H}(t)dt \\
        &= \frac{1}{\sqrt{2\pi}} \int_{-\infty} ^ {\infty} e ^ {-t ^ 2 / 2} \mu^+_{\overline{G \triangledown H}}(t)dt,
    \end{align*}
    which is the number of perfect matchings in $G \triangledown H$ \cite{Godsil}.
\end{proof}

The square of the norm of the matching polynomial can be expressed as the number of perfect matchings in the join graph.

\begin{crl} Define the norm induced by the inner product as follows:

\[
\left\Vert f \right\Vert = \sqrt{\langle f, f\rangle}.
\]

The square of the norm of a polynomial $\mu_{\bar G}^+(t)$ equals the number of perfect matchings of a join graph $G \triangledown G$:

\[
\left\Vert  \mu_{\bar G}^+ \right\Vert ^ 2 = m(G \triangledown G).
\]
\end{crl}

Thus, we obtain the following inequality for the number of perfect matchings in join graphs.

\begin{crl} Let $G$ and $H$ be simple graphs. The number of perfect matchings of join graphs satisfies the Cauchy-Schwarz inequality:

\[
m(G \triangledown H) ^ 2 \leq m(G \triangledown G) m(H \triangledown H).
\]
\end{crl}

\section{The clique-cover polynomial}

Section 5.1 derives the expression for the clique-cover polynomials of join graphs. In Section 5.2, we prove the duality theorem, which relates the clique-cover polynomial and the chromatic polynomial. Duality allows us to express the clique-cover polynomial and the chromatic polynomial of join graphs.

\subsection{Clique-cover polynomial of join graphs}
First, we prove the following lemma for complete multipartite graphs.

\begin{lm}\label{clique-pre}
    Define a distributive product $(\times)$ as $t ^ a \times t ^ b = \xi_{K_{a, b}}(t)$. Then, it is commutative and associative, and satisfies the following:
    
    \[
    t ^ {\alpha_1} \times t ^ {\alpha_2} \times \cdots t ^ {\alpha_n} = \phi_\xi \left[ \phi_\xi ^ {-1} \left[t ^ {\alpha_1} \right] \phi_\xi ^ {-1} \left[t ^ {\alpha_2} \right] \cdots \phi_\xi ^ {-1} \left[t ^ {\alpha_n} \right] \right],
    \]
    
    where the operator $\phi_\xi$ is defined implicitly with Stirling numbers of the second kind as follows:
    \begin{align*}
    \phi_\xi[t ^ n] &= \sum_{k = 1} ^ n \stirlings{n}{k} t ^ k, \\ \phi_\xi^{-1}[t ^ n] &= (t)_n,
    \end{align*}
    where $(t)_n = t (t - 1) \cdots (t - n + 1)$ is the falling factorial.  
\end{lm}

\begin{proof}
    Commutativity follows from the fact that $K_{a, b} \cong K_{b, a}$.
    \[
    t ^ a \times t ^ b = \xi_{K_{a, b}}(t) = \xi_{K_{b, a}}(t) = t ^ b \times t ^ a.
    \]
    Now, we show the following:
    \begin{equation}\label{xi_fundamental}
        \xi_{K_{\alpha_1, \alpha_2, \cdots, \alpha_n}}(t) \times t ^ {\alpha_{n + 1}} = \xi_{K_{\alpha_1, \alpha_2, \cdots, \alpha_{n + 1}}}(t).
    \end{equation}
    Expanding the left-hand side

    \begin{align*}
        \xi_{K_{\alpha_1, \alpha_2, \cdots, \alpha_n}}(t) \times t ^ {\alpha_{n + 1}} &= \left( \sum_{i = 0} ^ {\alpha_1 + \cdots + \alpha_n} s_i(K_{\alpha_1, \cdots, \alpha_n}) t ^ i \right) \times t ^ {\alpha_{n + 1}} \\ &=
        \sum_{i = 0} ^ {\alpha_1 + \cdots + \alpha_n} s_i(K_{\alpha_1, \cdots \alpha_n}) \xi_{K_{i, \alpha_{n + 1}}}(t),
    \end{align*}
    which can be interpreted as choosing $i$ cliques from $K_{\alpha_1, \cdots, \alpha_n}$ and attaching $\alpha_{n + 1}$ other points to either create new trivial cliques or attach them to some of the currently existing cliques. This is the same as the partitioning $K_{\alpha_1, \cdots, \alpha_{n + 1}}$ into cliques, which proves (\ref{xi_fundamental}).
    
    The product $(\times)$ is associative:

    \[
    (t ^ a \times t ^ b) \times t ^ c = \xi_{K_{a, b, c}}(t) = \xi_{K_{b, c, a}} = (t ^ b \times t ^ c) \times t ^ a =t ^ a \times (t ^ b \times t ^ c).
    \]
    Using (\ref{xi_fundamental}), one obtains

    \[
    t ^ {\alpha_1} \times t ^ {\alpha_2} \times \cdots \times t ^ {\alpha_n} = \xi_{K_{\alpha_1, \alpha_2, \cdots \alpha_n}}(t).
    \]

    Note that $t ^ {\times n} = t \times t \times \cdots \times t = \xi_{K_{1, 1, \cdots, 1}}(t) = \xi_{K_n}(t)$ forms a basis of the polynomial space. This is because the degree of polynomials $\xi_{K_n}(t)$ is $n$. Thus, they are linearly independent, and they span the whole space.
    
    By Lemma \ref{iso_prod}, there exists an operator $\phi_\xi$ such that

    \[
    t ^ {\alpha_1} \times t ^ {\alpha_2} \times \cdots \times t ^ {\alpha_n} = \phi_\xi \left[ \phi_\xi^{-1} [t ^ {\alpha_1}] \phi_\xi^{-1} [t ^ {\alpha_2}] \cdots \phi_\xi^{-1} [t ^ {\alpha_n}]\right].
    \]
    Note that 
    \[
    \phi_\xi[t ^ n] = t \times t \times \cdots \times t = \xi_{K_n}(t).
    \]
    Thus, one can express $\phi_\xi$ implicitly as follows:
    \[
    \phi_\xi[t ^ n] = \xi_{K_n}(t) = \sum_{k = 1} ^ {n} \stirlings{n}{k} t ^ k,
    \]
    which is true because the number of $k\text{-clique covers}$ in $K_n$ is $\stirlings{n}{k}$.
\end{proof}

The main theorem of this section is the formula for the clique-cover polynomial of the join graph.

\begin{thm}\label{clique-main}
    Let $G$ and $H$ be simple graphs. Then, the clique-cover polynomial $\xi_{G \triangledown H}(t)$ satisfies

    \[
    \xi_{G \triangledown H}(t) = \phi_\xi \left[ \phi_\xi ^ {-1}[\xi_G(t)] \phi_\xi ^ {-1}[\xi_H(t)]\right].
    \]
\end{thm}

\begin{proof}
    Let $n_1$ and $n_2$ be the number of vertices in $G$ and $H$, respectively. By definition, the coefficient of $t ^ k$ in $\xi_G(t)$ is the number of $k\text{-clique covers}$. If $\mathcal{A}$ is a clique cover of $G$ and $\mathcal{B}$ is a clique cover of $H$, then all clique covers of $K_{|\mathcal{A}|, |\mathcal{B}|}$ define clique covers in $G \triangledown H$. Conversely, all clique covers in $G \triangledown H$ can be expressed as clique covers in $K_{|\mathcal{A}|, |\mathcal{B}|}$ for some $\mathcal{A}$ and $\mathcal{B}$. Thus, the clique-cover polynomial can be written as

    \[
    \xi_{G \triangledown H}(t) = \sum_{l_1 = 0} ^ {n_1} \sum_{l_2 = 0} ^ {n_2} s_{l_1} (G) s_{l_2}(H) \xi_{K_{l_1, l_2}}(t).
    \]
    By Lemma \ref{clique-pre}, one has that $\xi_{K_{l_1, l_2}}(t) = t ^ {l_1} \times t^ {l_2} = \phi_\xi \left[ \phi_\xi ^ {-1} [t ^ {l_1}] \phi_\xi ^ {-1} [t ^ {l_2}]\right]$. Thus

    \begin{align*}
        \xi_{G \triangledown H}(t) &= \sum_{l_1 = 0} ^ {n_1} \sum_{l_2 = 0} ^ {n_2} s_{l_1} (G) s_{l_2}(H) \phi_\xi \left[ \phi_\xi ^ {-1} [t ^ {l_1}] \phi_\xi ^ {-1} [t ^ {l_2}]\right]\\
        &= \phi_\xi \left[ \sum_{l_1 = 0} ^ {n_1} \sum_{l_2 = 0} ^ {n_2} s_{l_1} (G) s_{l_2}(H)  \phi_\xi ^ {-1} [t ^ {l_1}] \phi_\xi ^ {-1} [t ^ {l_2}]
        \right] \\
        &= \phi_\xi \left[
        \phi^{-1}_\xi\left[\sum_{l_1 = 0} ^ {n_1} s_{l_1} (G)  t ^ {l_1} \right]
        \phi^{-1}_\xi\left[\sum_{l_2 = 0} ^ {n_2} s_{l_2}(H) t ^ {l_2} \right]
        \right] \\
        &= \phi_\xi \left[ \phi_\xi ^ {-1}[\xi_G(t)] \phi_\xi ^ {-1}[\xi_H(t)]\right].
    \end{align*}
\end{proof}

\subsection{Duality of the clique-cover polynomial and the chromatic polynomial}

It is possible to express the chromatic polynomial using clique covers of the graph complement (see \cite{Bollobas}, pg.151).

\begin{equation} \label{chrom-clique-duality}
    \chi_G(t) = \sum_{k = 1} ^ n c_k(\bar G) (t)_k.
\end{equation}

Define $\phi_\chi = \phi_\xi ^ {-1}$. Then, using  (\ref{chrom-clique-duality}), one obtains the duality relation between the clique-cover polynomial and the chromatic polynomial:

\[
\chi_G(t) = \phi_\chi \left[ \xi_{\bar G}(t) \right],
\]
\[
\xi_G(t) = \phi_\xi \left[ \chi_{\bar G}(t) \right].
\]

Equivalently, one can also state the relationship between the clique-cover polynomial and the chromatic polynomial as the expectation of the Poisson distribution.

\begin{prp} Let $X \sim \text{Pois($\lambda$)}$. Then, the clique-cover polynomial can be expressed as follows:
\[
\xi_{G}(\lambda) = \mathbb{E} \left[ \chi_{\bar G}(X) \right].
\]
\end{prp}

\begin{proof}
    It is well-known that, for any $k$, the expectation of the falling factorial satisfies

    \[
    \mathbb{E} \left[ (X)_k\right] = \lambda ^ k.
    \]
    Thus, using duality

    \[
    \mathbb{E} \left[ \chi_{\bar G}(X)\right] = \phi_{\xi} \left[ \chi_{\bar G}(\lambda)\right] = \xi_G(\lambda).
    \]
\end{proof}

We use the duality theorem to give a second proof of Theorem \ref{clique-main}.

\begin{proof}[Second proof of Theorem \ref{clique-main}]
    \begin{align*}
    \xi_{G \triangledown H}(t) &= \phi_{\xi}\left[\chi_{\overline{G\triangledown H}}(t)\right] = \phi_{\xi}\left[\chi_{{\bar G\cup \bar H}}(t)\right]
    \\ &= \phi_{\xi} \left[ \chi_{\bar{G}}(t) \chi_{\bar H}(t)\right] =
    \phi_{\xi} \left[ \phi_{\xi} ^ {-1} \left[ \xi_G(t) \right] \phi_{\xi} ^ {-1} \left[ \xi_H(t) \right] \right].
    \end{align*}
\end{proof}

Similarly, we obtain the formula for the chromatic polynomial of join graphs.

\begin{thm}
    Let $G$ and $H$ be simple graphs. Define the operator $\phi_{\chi}[t ^ n] = (t)_n$ as the falling factorial. Then, the chromatic polynomial of $G \triangledown H$ satisfies
    \[
    \chi_{G \triangledown H}(t) = \phi_\chi \left[ \phi_\chi ^ {-1}[\chi_G(t)] \phi_\chi ^ {-1}[\chi_H(t)]\right].
    \]
\end{thm}
\begin{proof}
    By duality
    \begin{align*}
    \chi_{G \triangledown H}(t) &= \phi_{\chi}\left[\xi_{\overline{G\triangledown H}}(t)\right] = \phi_{\chi}\left[\xi_{{\bar G\cup \bar H}}(t)\right] \\
    &= \phi_{\chi} \left[ \xi_{\bar{G}}(t) \xi_{\bar H}(t)\right] =
    \phi_{\chi} \left[ \phi_{\chi} ^ {-1} \left[ \chi_G(t) \right] \phi_{\chi} ^ {-1} \left[ \chi_H(t) \right] \right].
    \end{align*}
\end{proof}

\section{Applications}

In previous sections, we showed that the matching polynomials, path-cover polynomials, clique-cover polynomials, and chromatic polynomials satisfy duality relations. In this section, we outline the applications of our results.

\subsection{Graph Polynomials of Cographs}

Cographs are graphs with clique-width at most 2. Makowsky, Rotics, Averbouch, and Godlin \cite{Makowsky} gave an algorithm to compute matching and chromatic polynomials of graphs with bounded clique-width. In the context of cographs, their algorithm has polynomial time complexity.

For cographs, there exists a unified algorithm to compute these polynomials by only changing the operator. The main application of our results is the following algorithm for computing graph polynomials in cographs.

\begin{dfn}
    A cograph is a graph generated by a single vertex using join and union operations.

    \begin{itemize}
        \item $K_1$ is a cograph
        \item If $G = (V, E)$ and $H = (U, F)$ are cographs, then their graph union $G \cup H$ is also a cograph.
        \item If $G = (V, E)$ and $H = (U, F)$ are cographs, then their graph join $G \triangledown H$ is also a cograph.
    \end{itemize}

\end{dfn}

It was shown in \cite{Corneil} that any cograph $G$ can be uniquely represented as a tree, and it is possible to construct such trees with time complexity $O(N + E)$, where $N$ and $E$ are the numbers of vertices and edges in $G$, respectively. Let $T_G$ be a tree corresponding to the cograph $G$.

\begin{figure}[H]
    \centering
    \begin{tikzpicture}[scale=1,
      graphnode/.style={circle, fill=black, inner sep=1.2pt},
      leaf/.style={circle, fill=black, inner sep=1.2pt},
      joinnode/.style={circle, draw=black, fill=black, inner sep=2.4pt},
      unionnode/.style={circle, draw=black, fill=white, inner sep=2.4pt}
    ]
        % Cograph
        \node[graphnode, label=above:$1$] (g1) at (90:1.5) {};
        \node[graphnode, label=right:$2$] (g2) at (18:1.5) {};
        \node[graphnode, label=below right:$3$] (g3) at (306:1.5) {};
        \node[graphnode, label=below left:$4$] (g4) at (234:1.5) {};
        \node[graphnode, label=left:$5$] (g5) at (162:1.5) {};

        \draw
        (g1)--(g3) (g1)--(g4) (g1)--(g5)
        (g2)--(g3) (g2)--(g4) (g2)--(g5)
        (g3)--(g4) (g3)--(g5);

        \node at (0,-1.9) {$G$};
        % Cotree
        \node[joinnode] (r) at (5.6,-0.9) {};
        \node[leaf, label=above:$3$] (t3) at (4.0,0.2) {};
        \node[unionnode] (u12) at (5.6,0.2) {};
        \node[unionnode] (u45) at (7.2,0.2) {};
        \node[leaf, label=above:$1$] (t1) at (5.1,1.1) {};
        \node[leaf, label=above:$2$] (t2) at (6.1,1.1) {};
        \node[leaf, label=above:$4$] (t4) at (6.7,1.1) {};
        \node[leaf, label=above:$5$] (t5) at (7.7,1.1) {};

        \draw
        (r)--(t3) (r)--(u12) (r)--(u45)
        (u12)--(t1) (u12)--(t2)
        (u45)--(t4) (u45)--(t5);

        \node at (5.9,-1.9) {$T_G$};
    \end{tikzpicture}
    \caption{A cograph and its cotree.}
\end{figure}

The leaves of $T_G$ correspond to the vertices of $G$. Each black inner vertex represents a join operation. Each white inner vertex represents a union operation. The cograph associated with an internal vertex is obtained by applying the corresponding operation to the cographs induced by its child subtrees. For example, let $G$ be the graph obtained from $K_5$ by removing a matching of size two. Its cotree $T_G$ has five leaves and is rooted at a black vertex with three children: one leaf and two white vertices, each of which has two leaf children. The leaf labels in $T_G$ correspond to the vertex labels in $G$.

\begin{algorithm}
\caption{Compute graph polynomial for cograph}
\begin{algorithmic}[1]
\Function{ComputeGraphPolynomial}{$T_G, v, \phi$, $t$}
    \If{$v$ is a leaf}
        \State \Return $t$
    \EndIf
    \State result $\gets 1$
    \ForAll{$u \in \text{Children}(T_G, v)$}
        \If{$v$ is black}
            \State result $\gets \text{result} * \phi^{-1}\bigl[\text{ComputeGraphPolynomial}(T_G, u, \phi, t)\bigr]$
        \Else
            \State result $\gets \text{result} * \text{ComputeGraphPolynomial}(T_G, u, \phi, t)$
        
        \EndIf
    \EndFor
    \If{$v$ is black}
        \State result $\gets \phi[\text{result}]$
    \EndIf
    \State \Return result
\EndFunction
\end{algorithmic}
\end{algorithm}

\begin{thm}\label{algo}
    Let $P$ be a graph polynomial defined over simple graphs such that the degree of $P_G(t)$ is at most $\vert{}V(G)\vert{}$ for any graph $G$. Suppose $P$ is multiplicative, $P_{K_1}(t) = t$, and there exists an invertible linear operator $\phi$ such that for every pair of simple graphs $G$ and $H$, the following holds:
    
    \[P_{G \triangledown H}(t) = \phi \left[ \phi ^ {-1} \left[P_G(t) \right] \phi ^ {-1} \left[P_H(t) \right] \right]. \]
    
    Given a precomputed lookup table for $\phi[t^k]$ and $\phi^{-1}[t^k]$ for all $k \leq N = |V(G)|$, the graph polynomial of a cograph $G$ is $P_G(t) = \mbox{ComputeGraphPolynomial}(T_G, r, \phi, t)$ where $r$ is the root of $T_G$, and Algorithm 1 has time complexity $O(N ^ 3)$.
\end{thm}

\begin{proof}
    Let $v$ be a vertex of $T_G$. Denote $T_G(v)$ as the sub-tree of $T_G$ rooted at $v$, and let $G_v$ be the cograph formed by $T_G(v)$. If $v$ is a leaf, then $G_v \cong K_1$, from which it follows that $P_{G_v}(t) = t$. Suppose $v$ is not a leaf. Then, $v$ has a color and children.
    
    We denote the set of children of $v$ in $T_G$ as $\text{Children($T_G$, $v$)}$. If $v$ is black, then the graph polynomial of $G_v$ is the graph polynomial of a join graph formed by joining all $G_u$ for $u \in \text{Children($T_G$, $v$)}$. In other words

    \[
    P_{G_v}(t) = \phi \left[ \prod_{u \in \text{Children($T_G, v$)}} \phi ^ {-1}[P_{G_u}(t)]\right].
    \]

    By multiplicativity of the graph polynomial, if $v$ is white, then the graph polynomial of a cograph formed by $T_G(v)$ is

    \[
    \prod_{u \in \text{Children($T_G, v$)}}P_{G_u}(t).
    \]

    Therefore, if $r$ is the root of $T_G$, the graph polynomial $P_G(t)$ can be expressed as follows:

    \[
    P_G(t) = \mbox{ComputeGraphPolynomial}(T_G, r, \phi, t).
    \]
    
    We assume the arithmetic computational model, where basic arithmetic operations (addition, subtraction, multiplication, and division) take $O(1)$ time. Because $\phi$ is linear, it is sufficient to precompute and store the values $\phi [t ^ k]$ and $\phi ^ {-1} [t ^ k]$ for $k \leq N$ before executing Algorithm 1. Then, applying $\phi$ or $\phi^{-1}$ on a polynomial of degree $n$ is equivalent to replacing $t ^ k$ with $\phi[t ^ k]$ or $\phi ^ {-1}[t ^ k]$. Summing the polynomials that replaced the monomials after applying the operators results in a time complexity of $O(N ^ 2)$ because each polynomial has degree $O(N)$ and the degree of the original polynomial is $O(N)$. Thus, applying the operators on a polynomial has time complexity $O(N ^ 2)$. Because $T_G$ has $O(N)$ vertices, the algorithm applies the operators $O(N)$ times, from which it follows that the total contribution of operator applications is $O(N ^ 3)$. 
    
    Naive polynomial multiplication is $O(N ^ 2)$. Algorithm 1 performs the multiplication of polynomials $O(N)$ times because there are $O(N)$ vertices in $T_G$. Because the degrees of polynomials are at most $N$, the overall time complexity of Algorithm 1 is bounded by $O(N ^ 2) \cdot O(N) = O(N ^ 3)$.

\end{proof}

\begin{rmk}\label{algo_particular}
    For the graph polynomials considered in this paper, it is possible to execute Algorithm 1 without precomputing $\phi[t ^ k]$ and $\phi^{-1}[t ^ k]$. Applying the operators $\phi_\mu = e ^{D ^ 2 / 2}$, $\phi_\pi = e ^ {t D ^ 2}$, and their inverses is equivalent to the change of basis between the monomial basis and the Hermite and Laguerre polynomial bases. It was shown by Bostan, Salvy, and Schost that such conversions have $O(N \log N)$ computational complexity \cite{Bostan}. By using fast polynomial multiplication algorithms such as the Fast Fourier Transform, which have time complexity $O(N \log N)$ to multiply polynomials of degree $N$, it is possible to execute Algorithm 1 in $O(N ^ 2 \log N)$ time complexity. However, since the clique-cover and chromatic polynomials involve converting between the monomial basis and the falling factorial basis, the application of operators $\phi_\chi$ and $\phi_\xi$ has $O(N \log ^ 2 N)$ time complexity \cite{Bostan}. In this case, the time complexity of Algorithm 1 is $O(N ^ 2 \log ^ 2 N).$
\end{rmk}

%  Applying the operators $\phi_\mu = e ^{D ^ 2 / 2}$, $\phi_\pi = e ^ {t D ^ 2}$, and their inverses is equivalent to the change of basis between the monomial basis and the Hermite and Laguerre polynomial bases. It was shown by Bostan, Salvy, and Schost that such conversions are $O(N \log N)$ computational complexity \cite{Bostan}. However, since the clique-cover and chromatic polynomials involve converting between the monomial basis and the falling factorial basis, the conversions are $O(N \log ^ 2 N)$.

It is possible to compute the following for a cograph $G$ in $O(N^2 \log N)$ time.

\begin{enumerate}
    \item Number of perfect matchings, $\mu_G(0)$,
    \item Number of directed Hamiltonian paths, $\left. \frac{d}{dt}\pi_G(t) \right|_{t = 0},$
    \item Number of directed Hamiltonian cycles, Corollary \ref{hc_char_0}, if $G$ is connected
\end{enumerate}

The following can be computed in $O(N ^ 2 \log ^ 2 N)$ time for a cograph $G$.

\begin{enumerate}
    \item Number of proper colorings in $\lambda$ colors, $\chi_G(\lambda)$,
    \item Number of acyclic orientations \cite{Stanley}, $|\chi_G(-1)|.$
\end{enumerate}

\subsection{Hamiltonian Paths in Complete $m\text{-partite}$ graphs}

By Theorem \ref{path-main}, the number of directed Hamiltonian paths in $K_{\alpha_1, \alpha_2, \cdots, \alpha_m}$ is

\begin{equation}\label{hampath}
    [t]\left(\phi_\pi \left[ 
    \prod_{i = 1} ^ m \phi_\pi ^ {-1} \left[
    t ^ {\alpha_i}
    \right]
    \right]
    \right). 
\end{equation}

By Proposition \ref{ham_integral}, if $X \sim \text{Exp}(1)$, one can also express (\ref{hampath}) using associated Laguerre polynomials.

\begin{thm} The number of directed Hamiltonian paths in $K_{\alpha_1, \alpha_2, \cdots, \alpha_m}$ is given by
    \[
    \mathbb{E}\left[ 
        \prod_{i = 1} ^ m (-1) ^ {\alpha_i} \alpha_i!L_{\alpha_i}^{(-1)}(X)
        \right].
    \]
\end{thm}

Define the partial Bell polynomial $B_{n, k}(x_1, x_2, \cdots, x_{n - k + 1})$ as

\[
B_{n, k}(x_1, x_2, \cdots, x_{n - k + 1}) = \sum_{\alpha} \frac{n!}{\alpha_1! \alpha_2! \cdots \alpha_{n - k + 1}!} \prod_{j = 1} ^ {n - k + 1}\left(\frac{x_j}{j!}\right) ^ {\alpha_j},
\]

where the sum is taken over all sequences $\alpha$ of non-negative integers satisfying following two conditions.
    
\begin{enumerate}
    \item $\alpha_1 + \alpha_2 + \alpha_3 + \cdots + \alpha_{n - k + 1} = k$,
    \item $\alpha_1 +2 \alpha_2 + 3\alpha_3 + \cdots + (n - k+ 1) \alpha_{n - k + 1} = n$.
\end{enumerate}

We use (\ref{hampath}) to obtain an explicit formula for the number of directed Hamiltonian paths in the complete $m\text{-partite graph}$ $K_{n, n, \cdots, n}$.

\begin{prp}\label{HP_n-n-...-n}
    The number of directed Hamiltonian paths in the complete $m\text{-partite graph}$ $K_{n, n, \cdots, n}$ is
    % \[
    % \mathbb{E} \left[ \left( 
    % \sum_{k = 1} ^ n \frac{X ^ k}{k!} (-1) ^ {n - k} n! \binom{n - 1}{k - 1}
    % \right) ^ m \right] = 
    % \]
    % \[\int_{0} ^ {\infty} e ^{-x} \left( 
    % \sum_{k = 1} ^ n \frac{x ^ k}{k!} (-1) ^ {n - k} n! \binom{n - 1}{k - 1}
    % \right) ^ m dx=
    % \]
    \[
    (n!) ^ m m! (-1) ^ {n m} \sum_{k = m} ^ {n m} (-1) ^ k B_{k, m} \left( 
    \left\{ \binom{n - 1}{i - 1} \right\}_{i = 1} ^ {k  - m + 1}
    \right).
    \]
\end{prp}

\begin{proof}
    Using (\ref{hampath}), one obtains the number of directed Hamiltonian paths in $K_{n, n, \cdots, n}$:
    \[
    [t] \left( \phi_\pi \left[ \left(
    \phi_\pi ^ {-1} \left[ 
    t ^ n
    \right] \right) ^ m
    \right]
    \right) = 
    \]
    \[
    ([t]  \phi_\pi) \left( \left[
    \left( 
    \sum_{k = 1} ^ n (-1) ^ {n - k} \frac{t ^ k}{k!} n! \binom{n - 1}{k - 1}
    \right) ^ m
    \right] \right).
    \]
    We expand the inner exponential term using the partial Bell polynomial (see \cite{Comtet}, p. 133):
    \[
    (n!) ^ m ([t] \phi_\pi) \left[ 
    m! \sum_{k = 1} ^ {n m} \frac{t ^ k}{k!} B_{k, m}\left(\left\{(-1) ^ {n - i}\binom{n - 1}{i - 1} \right\}_{i = 1} ^ {k - m + 1} \right)
    \right]
    = 
    \]
    \[
    (n!) ^ m m! \sum_{k = 1} ^ {n m} ([t] \phi_\pi) \left[
    \frac{t ^ k}{k!} B_{k, m}\left(\left\{(-1) ^ {n - i}\binom{n - 1}{i - 1} \right\}_{i = 1} ^ {k - m + 1} \right)
    \right].
    \]
    The operator $[t] \phi_\pi$ is the operator that replaces $t^k$ with $k!$. Then
    \[
    ([t] \phi_\pi)[t ^ n] = [t] \sum_{k = 1} ^ n \frac{n!}{k!} \binom{n - 1}{k - 1} t ^ k = n!. 
    \]
    Therefore 
    \begin{align*}
    &(n!) ^ m m! \sum_{k = 1} ^ {n m}
     B_{k, m}\left(\left\{(-1) ^ {n - i}\binom{n - 1}{i - 1} \right\}_{i = 1} ^ {k - m + 1} \right) \\ &= 
    (n!) ^ m m! (-1) ^ {nm} \sum_{k = 1} ^ {n m}
     (-1) ^ k B_{k, m}\left(\left\{\binom{n - 1}{i - 1} \right\}_{i = 1} ^ {k - m + 1} \right). 
    \end{align*}

% So, on one hand, if $X \sim \text{Exp(1)}$, the expression can be written as the expectation of exponential distribution:

    % \[
    % \mathbb{E} \left[ \left( 
    % \sum_{k = 1} ^ n \frac{X ^ k}{k!} (-1) ^ {n - k} n! \binom{n - 1}{k - 1}
    % \right) ^ m \right] = 
    % \]
    % \[\int_{0} ^ {\infty} e ^{-x} \left( 
    % \sum_{k = 1} ^ n \frac{x ^ k}{k!} (-1) ^ {n - k} n! \binom{n - 1}{k - 1}
    % \right) ^ m dx.
    % \]
\end{proof}

\subsection{Hamiltonian cycles in Complete $(m+1)\text{-partite graph}$} 

By Theorem \ref{cycle_characteristic}, it is possible to formulate the number of directed Hamiltonian cycles in terms of associated Laguerre polynomials.

\begin{thm} Let $X \sim \text{Exp(1)}$. The number of directed Hamiltonian cycles in $K_{\alpha_1, \alpha_2, \cdots, \alpha_{m + 1}}$ is
    \[
    \mathbb{E} \left[ \frac{1}{X} \prod_{i = 1} ^ {m + 1}(-1) ^ {\alpha_i} \alpha_i! L_{\alpha_i} ^ {(-1)} (X) \right].
    \]
\end{thm}

A recurrence relationship for the number of directed Hamiltonian cycles in $K_{\alpha_1, \alpha_2, \cdots, \alpha_{m + 1}}$ was obtained in \cite{Horak}. Computationally, the formula in \cite{Horak} works in $O(n ^m)$ time. By Theorem \ref{algo}, our formula works in $O((nm)^3)$ time and can be optimized to $O((nm) ^ 2\log(nm))$ according to Remark \ref{algo_particular}.

By Proposition \ref{HC}, the number of directed Hamiltonian cycles in $K_{\alpha_1, \alpha_2, \cdots, \alpha_{m + 1}}$ can also be expressed as the following:

\[
\sum_{k \geq 2} k! (k - 1)! \left(\left[ t ^ k \right] \phi_\pi \left[\prod_{i = 1} ^ {m} \phi^{-1}_\pi [t ^ {\alpha_i}]\right] \right) \left(\left[ t ^ k \right] t ^ {\alpha_{m + 1}}\right) = 
\]
\begin{equation}\label{HC_m-partite}
    \alpha_{m + 1}! (\alpha_{m + 1} - 1)! \left(\left[ t ^ {\alpha_{m + 1}} \right] \phi_\pi \left[\prod_{i = 1} ^ {m} \phi^{-1}_\pi [t ^ {\alpha_i}]\right] \right).
\end{equation} 
When $\alpha_1 = \alpha_2 = \cdots = \alpha_{m + 1} = n$, we obtain an explicit formula. 

\begin{prp}\label{HC_n-n...-n}

The number of directed Hamiltonian cycles in the complete $(m + 1)\text{-partite}$ graph $K_{n, n, \cdots, n}$ is

 \[
    (-1) ^ {n m} (n!) ^ m (n - 1)! m! \sum_{k = n} ^ {nm} (-1) ^ k \binom{k - 1}{n - 1} B_{k, m} \left( 
    \left\{ \binom{n - 1}{i - 1} \right\}_{i = 1} ^ {k  - m + 1}
    \right).
    \]
\end{prp}

\begin{proof}
    Setting $\alpha_1 = \alpha_2 = \cdots = \alpha_{m + 1} = n$ in (\ref{HC_m-partite}), one obtains the following:
    \begin{equation}\label{HC_fast}
        n! (n - 1)! \left[ t ^ n \right] \phi_\pi \left[ \left( \sum_{k = 1} ^ n (-1) ^ {n + k} \frac{t ^ k}{k!} n! \binom{n - 1}{k - 1}\right) ^ m \right].
    \end{equation}
    The inner exponential term can be expanded in terms of the partial Bell polynomial (see \cite{Comtet}, p. 133):
    \[
    m! (n!) ^ {m + 1} (n - 1)! \left[ t ^ n \right]  \phi_\pi \left[ 
    \sum_{k = 1} ^ {nm} \frac{t ^ k}{k!} B_{k, m}\left(\left\{(-1) ^ {n - i} \binom{n - 1}{i - 1}\right\}_{i = 1} ^ {k - m + 1}\right)
    \right].
    \]
    Applying $\phi_\pi$ on $t ^ k$, we get
    \[
    m! (n!) ^ {m + 1} (n - 1)! \left[ t ^ n \right] \sum_{k = 1} ^ {nm} \frac{1}{k!} \left( \sum_{l = 1} ^ {k} \frac{k!}{l!} \binom{k - 1}{l - 1} t ^ l \right) B_{k, m}\left(\left\{(-1) ^ {n - i} \binom{n - 1}{i - 1}\right\}_{i = 1} ^ {k - m + 1}\right).
    \]
    Thus, taking the coefficient of $t ^ n$, we obtain
    \begin{align*}
    &m! (n!) ^ {m + 1} (n - 1)! \sum_{k = 1} ^ {nm}  \frac{1}{n!} \binom{k - 1}{n - 1} B_{k, m}\left(\left\{(-1) ^ {n - i} \binom{n - 1}{i - 1}\right\}_{i = 1} ^ {k - m + 1}\right)\\ &=
     (-1) ^ {n m} (n!) ^ m (n - 1)! m! \sum_{k = n} ^ {nm} (-1) ^ k \binom{k - 1}{n - 1} B_{k, m} \left( 
    \left\{ \binom{n - 1}{i - 1} \right\}_{i = 1} ^ {k  - m + 1}
    \right).
    \end{align*}
\end{proof}

The expressions in Propositions \ref{HP_n-n-...-n} and \ref{HC_n-n...-n} can be evaluated in $O(n m ^ 2)$ time, by computing the Bell polynomial in $O(n m ^ 2)$. A more efficient way to calculate the number of directed Hamiltonian cycles is to use (\ref{HC_fast}). Namely, one can use fast polynomial multiplication, such as the Fast Fourier Transform, and binary exponentiation to reduce the time complexity to $O((nm) \log(nm) \log(m))$.

\section*{Acknowledgments}

The author is grateful to the anonymous referees for their very careful reading and valuable suggestions. This research was funded by the Science Committee of the Ministry of Science and Higher Education of the Republic of Kazakhstan (Grant No. BR 28713025).

\printbibliography

\end{document}